\documentclass[12pt]{amsart}
\usepackage{amsmath,mathtools,amsthm,amsfonts,bigints,amssymb, mathrsfs, tikz-cd} 

\newtheorem{theorem}{Theorem}[section]
\newtheorem{lemma}[theorem]{Lemma}

\theoremstyle{definition}
\newtheorem{definition}[theorem]{Definition}

\newtheorem{remark}[theorem]{Remark}

\newcommand{\R}{\mathbb R}%
\newcommand{\N}{\mathbb N}%
\newcommand{\E}{\mathbb E}%

\numberwithin{equation}{section}
\makeatletter

\renewcommand\subsubsection{\@secnumfont}{\bfseries}%
\renewcommand\subsubsection{\@startsection{subsubsection}{3}
  \z@{.5\linespacing\@plus.7\linespacing}{-.5em}%
  {\normalfont\bfseries}}
  
  \makeatother

\makeatletter
\@namedef{subjclassname@2020}{%
  \textup{2020} Mathematics Subject Classification}
\makeatother 

\begin{document}

\title[Restricted Mean Value Property]{Restricted Mean Value Property on Riemannian manifolds}

\author[K. Biswas]{Kingshook Biswas}
\address{Stat-Math Unit, Indian Statistical Institute, 203 B. T. Rd., Kolkata 700108, India}
\email{kingshook@isical.ac.in}

\author[U. Dewan]{Utsav Dewan}
\address{Stat-Math Unit, Indian Statistical Institute, 203 B. T. Rd., Kolkata 700108, India}
\email{utsav\_r@isical.ac.in}



\begin{abstract} 
A well studied classical problem is the harmonicity of functions satisfying the restricted mean-value property (RMVP). While this has so far been studied mainly for domains in $\mathbb{R}^n$, we consider this problem in the general setting of  domains in Riemannian manifolds, and obtain results generalizing classical results of Fenton. We also obtain a result for complete, simply connected Riemannian manifolds of pinched negative curvature where there is no restriction on the radius function in the RMVP.
\end{abstract} 

\maketitle
\tableofcontents

\section{Introduction}
We recall that a continuous function in a domain $D$ in $\R^n$ is harmonic if and only if it satisfies the spherical mean value property (with respect to the surface volume measure) for all spheres contained in $D$. A classical problem which has been considered by various authors is whether harmonicity still holds if one assumes a much weaker version of the mean-value property. Namely, a continuous function on a domain $D$ in $\mathbb{R}^n$ is said to satisfy the {\it restricted mean-value property} (RMVP for short) if for each point $x \in D$, there exists a sphere $S$ of center $x$ and some radius $\rho(x)$, which is contained in the domain $D$, such that $u(x)$ equals the mean value of $u$ on the sphere $S$. One can then ask whether a function satisfying the restricted mean value property in $D$ is harmonic in $D$.

\medskip

For a bounded domain $D$, it turns out that the key factor in an answer to this question is the boundary behaviour of the function. Indeed, if one assumes that the function $u$ extends continuously to the closure of the domain $D$, then  Kellogg gave a simple argument to show that the function must be harmonic (\cite{Kellogg}). Without any assumptions on the boundary behaviour, there are counter-examples, which are unbounded (see page 22, \cite{Littlewood}). In this context, Littlewood asked (in \cite{Littlewood}) for $n=2$, whether the unboundedness is the only obstruction. In other words, if the function satisfying the RMVP is also assumed to be bounded, then is it harmonic? This is the classical `one-circle problem' and in \cite{HN2}, Hansen and Nadirashvili showed that the above problem has a negative answer, that is, there exists a continuous, bounded function on the unit disk in $\R^2$ which satisfies RMVP but is not harmonic. The problem is still open for $n \ge 3$ however. Nevertheless, obtaining sufficient conditions for a function satisfying RMVP to be harmonic has received considerable attention over the years, 
in particular from Fenton and Hansen-Nadirashvili (\cite{Fenton,Fenton2,Fenton3,HN1}).

\medskip

While these authors considered this problem for domains in $\mathbb{R}^n$, in 
the present article we consider this problem in the very general setting of a domain in a Riemannian manifold. The definition of the RMVP in this setting needs some explanation. In a Riemannian manifold, it is not true in general that harmonic functions satisfy the mean-value property with respect to mean-values over geodesic spheres when the mean-value is taken with respect to the surface volume measure induced by the Riemannian metric (although there is a special class of Riemannian manifolds for which this holds, namely the {\it harmonic manifolds}, 
see \cite{Willmore}). The mean-value property does hold however, if one replaces the Riemannian measures on the geodesic spheres by the {\it harmonic measures} on the spheres. These are certain measures on the spheres (or more generally on the boundaries of precompact domains with smooth boundary in the manifold) which are mutually absolutely continuous with respect the Riemannian surface volume measures. For the definition of the harmonic measures we refer to section 2.  

\medskip

The harmonic measures have the following property. For any harmonic function $u$ on a domain $D$ in a Riemannian manifold $M$ and any geodesic ball $B$ such that $\overline{B} \subset D$, one has the following mean value property
\begin{equation} \label{mvp}
u(z) = \int_{\partial B} u(\xi) d\mu_{z,B}(\xi) \:,
\end{equation}  
where $z$ is the center of $B$, and $\mu_{z,B}$ is the harmonic measure on $\partial B$ with respect to $z$. 

\medskip

Thus harmonic functions on a Riemannian manifold do satisfy the mean-value property, if one takes mean-values with respect to harmonic measures. This leads us to the following definition:
\begin{definition} \label{defn_rmvp} {\bf (Restricted Mean-Value Property)}
Let $D$ be a domain in a Riemannian manifold $M$. 
 A continuous function $u$ on $D$ is said to satisfy the Restricted Mean Value Property in $D$ if for all $z \in D$, there exists $0<\rho(z)<inj(z)$ (where $inj(z)$ is the injectivity radius of $z$) such that the closed ball $B = \overline{B(z,\rho(z))}$ is contained in $D$, and the equality (\ref{mvp}) holds.   
\end{definition}

We now let $M$ be a Riemannian manifold, and let $D \subset M$ be any precompact domain in $M$ with smooth boundary, such that the distance to the boundary $d(z,\partial D)$ is smaller than the injectivity radius $inj(z)$, for all points $z$ in $D$. 

\medskip

In this setting we have the following result, which is a vast generalization of Theorem 2 in \cite{Fenton} (\cite{Fenton} only considers the case of the unit disk in $\mathbb{R}^2)$:

\medskip
 
\begin{theorem} \label{domain_bdry_bhv}
If $u: D \subset M \to \R$ is bounded, continuous, satisfies the Restricted Mean Value Property in $D$ and 
\begin{equation} \label{domain_bdry_bhv_eqn} 
\displaystyle\lim_{\substack{z \to \xi\\ z \in D}} u(z) = u_\xi \:\text{exists for almost every }\xi \in \partial D \:,
\end{equation}
with respect to the Riemannian measure on $\partial D$, then  $u$ is harmonic in $D$.
\end{theorem}

Our next result is a generalization of Theorem 1 from \cite{Fenton}. This result  does not assume boundedness nor any explicit boundary behaviour of $u$, but rather assumes a suitable rate of decay of the radius function $\rho(z)$ as $z$ tends 
to the boundary of $D$ in terms of the modulus of continuity of $u$ on compacts in $D$. In order to state it, we introduce some notation.

\medskip

For a pre-compact domain $D$ define,
\begin{equation} \label{defn_radius}
r_D := \sup \{d(x,\partial D): x \in D\} \:.
\end{equation}
Then for a continuous function $u$ on $D$, given $\varepsilon >0$ and given $r \in (0,1)$, let $\delta(r,\varepsilon;u)$ be the largest positive number such that 
\begin{equation*}
|u(z)-u(y)| \le \varepsilon 
\end{equation*}
holds, provided
\begin{equation*}
d(z,\partial D) \ge (1-r)r_D \text{ and } d(z,y) \le  \delta(r,\varepsilon;u) \:.
\end{equation*}

Define a function $r$ on $D$ by $r(z) = 1-\left(d\left(z,\partial D\right)/ 2\:r_D\right)$ for $z \in D$, and note that for any $z \in D$ we have
$d(z, \partial D) \geq (1 - r(z))r_D$.

We then have:

\begin{theorem} \label{domain_radius}
Let $\kappa \in (0,1), \tau \in (0,1)$ and let $\varepsilon : (0,(1-\tau)r_D] \to (0, \infty)$ be a continuous function such that $\varepsilon(r) \to 0 \text{ as } r \to 0$. Let $u$ be a continuous function on $D$ such that $u$ satisfies the Restricted Mean Value Property at each point $z \in D$ on a geodesic sphere of radius $\rho(z)$, where 
\begin{equation*} 
\rho(z) \le \begin{cases} 
       \kappa \: d(z,\partial D) & \text{ if } d(z,\partial D) \ge (1-\tau)r_D \\
      \min\left\{\kappa \: d(z,\partial D)\:,\delta(r(z),\varepsilon(d(z,\partial D));u)\right\} & \text{ if } d(z,\partial D) < (1-\tau)r_D \:.
   \end{cases}
\end{equation*}
Then $u$ is harmonic in $D$.
\end{theorem}

Both results above follow the general scheme of the arguments in \cite{Fenton}, 
namely constructing subharmonic and superharmonic functions $v$ and $w$ respectively such that $w \leq u \leq v$, and then using the Maximum Principle to show that the nonnegative subharmonic function $v - w$ vanishes identically. 

\medskip

The two main ingredients required to carry out the above scheme are the following:

\medskip

\noindent (1) Concentration of harmonic measures: 

\medskip

This states that if we have a sequence of points $x_n$ in balls $B_n$, such that the balls $B_n$ converge to a ball $B$ and the points $x_n$ converge to a point $\xi$ on the boundary of $B$, then the harmonic measures $\mu_{x_n, B_n}$ on the boundaries $\partial B_n$, with respect to the basepoints $x_n$, converge weakly to the Dirac mass $\delta_{\xi}$ at $\xi$. 

\medskip

Note that if the sequence of balls is a constant sequence, $B_n = B$ for all $n$, then this fact follows immediately from the definition of harmonic measures in terms of solutions to the Dirichlet problem. For a general nonconstant sequence of converging balls this is somewhat non-trivial however. In Euclidean space this concentration of measure is easy to prove using the explicit formula for the Poisson kernels of the balls $B_n$. In a general Riemannian manifold, we are faced with the lack of any explicit formulae for the Poisson kernel of geodesic balls, and instead we use softer techniques, namely estimating harmonic measures by constructing superharmonic barriers.

\medskip

\noindent (2) Convergence of Poisson integrals: 

\medskip

This states that if $u$ is bounded and we have a sequence of balls $B_n$ in $D$ converging to a ball $B$ in $D$, then the Poisson integrals of $u$ on $B_n$ converge pointwise on $B$ to the Poisson integral of $u$ on $B$.

\medskip

Again, in Euclidean space this follows easily from the explicit formula for the Poisson kernel of a finite ball. In the case of a general Riemannian manifold, we prove this by writing the Poisson integral in terms of Brownian motion, and then applying the Dominated Convergence Theorem on the sample space of Brownian motion, i.e. the space of continuous paths in the manifold.  

\medskip

We also have a result for when the boundary is "at infinity". Namely, if we take  $M$ to be a complete, simply connected Riemannian manifold of pinched negative curvature, then it has infinite injectivity radius at each point, and we can consider functions on the whole manifold $M$ satisfying definition \ref{defn_rmvp}, where the domain $D = M$. In the previous two results where only bounded domains were considered, the radius function $z \mapsto \rho(z)$ was bounded. Instead if we take the domain $D$ to be the whole manifold $M$, 
then it is natural to ask for an analogue of Theorem \ref{domain_bdry_bhv} for the whole manifold $M$, where the radius function is allowed to be unbounded. 

\medskip

In fact, we do have a result in this case too, with no restriction at all on the radius function. In this case however, in place of the boundary $\partial D$, one needs to consider the {\it Gromov boundary} $\partial M$ of the Gromov hyperbolic space $M$, and the appropriate measure class on the Gromov boundary in this context is the harmonic measure class (see section 2 for the definitions of Gromov boundary and harmonic measures on the Gromov boundary). 

\medskip

We then have the following:

\medskip

\begin{theorem} \label{neg_bdry_bhv}
Let  $M$ be a complete, simply connected Riemannian manifold of dimension $\ge 2$, satisfying sectional curvature bounds $-b^2 \le K_M \le -a^2$, for some $0<a\le b$. Let $u: M \to \R$ be a bounded, continuous function which satisfies the Restricted Mean Value Property. Suppose 
\begin{equation} \label{neg_bdry_bhv_eqn} 
\displaystyle\lim_{\substack{z \to \xi\\ z \in M}} u(z) = u_\xi\:,\text{exists for almost every }\xi \in \partial M\:,
\end{equation}
with respect to the harmonic measures on $\partial M$. 

Then  $u$ is harmonic in $M$.
\end{theorem}

\medskip

We remark that any Riemannian symmetric space of noncompact type of rank one satisfies the conditions of Theorem \ref{neg_bdry_bhv}.

\medskip

The proof of the above Theorem involves some new ingredients related to the geometry of negatively curved manifolds. In particular, allowing the radius function $\rho$ to be unbounded necessitates an understanding of horoballs and how geodesic balls converge to horoballs. We also need an $L^\infty$ version of the maximum principle for subharmonic functions, concerning boundary values almost everywhere on the Gromov boundary with respect to harmonic measure, which we prove using Brownian motion.

\medskip

The paper is organised as follows. In section $2$, we discuss the relevant preliminaries regarding Gromov hyperbolic spaces, their boundaries, and harmonic measures. In section $3$, we prove the crucial results on concentration of harmonic measures and convergence of Poisson integrals. In section $4$, we construct an auxiliary subharmonic function which will play a central role in the proofs of our main results. In section $5$, we prove some useful $L^\infty$ maximum principles. In section $6$, we give proofs of Theorems \ref{domain_bdry_bhv},
\ref{domain_radius} and  \ref{neg_bdry_bhv}. Finally, we conclude by making some remarks in section $7$. 

\medskip

\section{Preliminaries} 

\medskip

We recall briefly some basic properties of Gromov hyperbolic spaces and CAT(k) spaces, for more details we refer to \cite{Bridson}. 

\medskip

A {\it geodesic} in a metric space $X$ is an isometric embedding $\gamma : I \subset \mathbb{R} \to X$ of an interval into $X$. The metric space $X$ is said to be {\it geodesic} if any two points in $X$ can be joined by a geodesic. A geodesic metric space $X$ is said to be {\it Gromov hyperbolic} if there is a $\delta \geq 0$ such that every geodesic triangle in $X$ is $\delta$-thin, i.e. each side is contained in the $\delta$-neighbourhood of the union of the other two sides.

\medskip

The {\it Gromov boundary} of a Gromov hyperbolic space $X$ is defined to be the set $\partial X$ of equivalence classes of geodesic rays in $X$. Here a geodesic ray is an isometric embedding $\gamma : [0,\infty) \to X$ of a closed half-line into $X$, and two geodesic rays $\gamma_1, \gamma_2$ are said to be equivalent if the set $\{ d(\gamma_1(t), \gamma_2(t)) \ | \ t \geq 0 \}$ is bounded. The 
equivalence class of a geodesic ray $\gamma$ is denoted by $\gamma(\infty) \in \partial X$. 

\medskip

A metric space is said to be {\it proper} if closed and bounded balls in the space 
are compact. Let $X$ be a proper, geodesic, Gromov hyperbolic space. There is a natural topology on $\overline{X} := X \cup \partial X$, called the {\it cone topology} such that $\overline{X}$ is a compact metrizable spacew which is a compactification of $X$. In this case, for every geodesic ray $\gamma$,  $\gamma(t) \to \gamma(\infty) \in \partial X$ as $t \to \infty$, and for any $x \in X, \xi \in \partial X$ there exists a geodesic ray  
$\gamma$ such that $\gamma(0) = x, \gamma(\infty) = \xi$. We will refer to neighbourhoods of $\xi \in \partial X$ with respect to the cone topology as {\it cone neighbourhoods}.

\medskip

For three points $x,y,z \in X$, the {\it Gromov inner product} of $y,z$ 
with respect to $x$ is defined to be   
$$
(y|z)_x := \frac{1}{2}(d(x,y)+d(x,z) - d(y,z)).
$$

If the space $X$ is in addition a CAT(k) space for some $k < 0$ (we refer to \cite{Bridson} for the definition of CAT(k) space), then for any $x \in X$, the Gromov inner product $(.|.)_x : X \times X \to [0,+\infty)$ extends to a continuous function $(.|.)_x : \overline{X} \times \overline{X} \to [0,+\infty]$, such that $(\xi|\eta)_x = +\infty$ if and only if $\xi = \eta \in \partial X$.
This holds in particular if $X$ is a complete, simply connected Riemannian manifold of sectional curvature bounded above by $k$ for some $k < 0$, as it is well known by a Theorem of Alexandrov that such a manifold is a CAT(k) space. In this case the compactification $\overline{X}$ is homeomorphic to the closed unit ball $\overline{\mathbb{B}} \subset \mathbb{R}^n$, and there is a homeomorphism 
$\phi : \overline{\mathbb{B}} \to \overline{X}$ such that the restriction to the open unit ball $\phi : \mathbb{B} \to X$ is a diffeomorphism.

\medskip

We now assume that $X$ is such a negatively curved manifold, i.e. $X$ is complete, simply connected and with sectional curvature bounded above by $-a^2$ for some $a > 0$. The {\it Busemann cocycle} of $X$ is the function $B : X \times X \times \partial X \to \mathbb{R}$ defined by
$$
B(x, y, \xi) := \lim_{z \to \xi} (d(x,z) - d(y,z)).
$$

The limit above exists, and the Busemann cocycle is a continuous function on $X \times X \times \partial X$. Fix a 
basepoint $o \in X$. Then for any $\xi \in \partial X$ and $r \in \mathbb{R}$, the {\it horoball based at $\xi$ of radius $r$} is defined to be the set
$$
H(\xi, r) := \{ x \in X \ | \ B(x,o,\xi) < r \ \}.
$$

The boundary of the horoball in $X$ is given by the set 
$$
\partial H(\xi, r) = \{ x \in X \ | \ B(x, o, \xi) = r \ \}
$$ 
and is called a {\it horosphere based at $\xi$ of radius $r$}. The closures of the horoball and the horosphere in $\overline{X}$ are given by the compact sets $H(\xi, r) \cup \partial H(\xi, r) \cup \{\xi\}$ and $\partial H(\xi, r) \cup \{\xi\}$ respectively. 

\medskip

From the Busemann cocycle, one can also define functions on $X$ called {\it Busemann functions}. For this we fix a basepoint $o \in X$. Then for any $\xi \in \partial X$, we can define a function $B_{\xi} : X \to \mathbb{R}$, called the {\it Busemann function based at $\xi$}, which is given by
$$
B_{\xi}(x) = B(x,o,\xi).
$$

It is well-known that the Busemann functions $B_{\xi}$ are $C^2$, and the gradient of the Busemann function has constant norm equal to one, $||\nabla B_{\xi}|| \equiv 1$. Note that the horospheres based at $\xi$, $\partial H(\xi, r), r \in \mathbb{R}$, are the level sets of the Busemann function $B_{\xi}$. Since the Busemann function is $C^2$ with nowhere vanishing gradient, it follows that the horospheres are $C^2$ submanifolds of $X$.

\medskip

We will be needing some well-known results on the convergence of balls to horoballs and of horoballs to horoballs. For any $x \in X$, let $d_x : X \to \mathbb{R}$ denote the distance function from the point $x$, given by 
$d_x(y) = d(x,y), y \in X$. Then for any $\xi \in \partial X$, if $x_n \in X$ is a sequence such that $x_n \to \xi$ as $n \to \infty$, then the functions $d_{x_n} - d(o,x_n)$ converge in $C^2$ norm on compacts to the Busemann function $B_{\xi}$. Also, if $\xi_n \in \partial X$ is a sequence such that $\xi_n \to \xi$ as $n \to \infty$, then the Busemann functions $B_{\xi_n}$ converge in $C^2$ norm on compacts to the Busemann function $B_{\xi}$.  

\medskip
 
It is well-known that if $x_n \in X$ is a sequence such that $x_n \to \xi \in \partial X$ as $n \to \infty$, and if $r_n > 0$ is a sequence such that 
$r_n - d(x_n, o) \to r \in \mathbb{R}$ as $n \to \infty$, then the closed balls $\overline{B}(x_n, r_n)$ converge to the closure of the horoball $H(\xi, r)$ in $\overline{X}$ with respect to the Hausdorff topology on compact subsets of $\overline{X}$. Similarly, if $\xi_n \in \partial X$ and $r_n \in \mathbb{R}$ are sequences such that $\xi_n \to \xi, r_n \to r$ as $n \to \infty$, then the closures in $\overline{X}$ of the horoballs $H(\xi_n, r_n)$ converge to the closure in $\overline{X}$ of the horoball $H(\xi, r)$, with respect to the Hausdorff topology on compact subsets of $\overline{X}$. 

\medskip 
   
We also need to recall the definition of harmonic measures. These arise from the solution of the Dirichlet problem. Given a precompact domain $D$ with smooth boundary in a Riemannian manifold, for any $x \in D$ the {\it harmonic measure on $\partial D$ with respect to $x$} is the probability measure $\mu_{x, D}$ on $\partial D$ defined by
$$
\int_{\partial D} f \ d\mu_{x, D} = u_f(x)
$$
for all continuous functions $f$ on $\partial D$, where $u_f$ is the solution of the Dirichlet problem in $D$ with boundary value $f$.  The harmonic measures $\mu_{x,D}$ are mutually absolutely continuous, in fact they are absolutely continuous with respect to the Lebesgue measure of $\partial D$. The harmonic measure $\mu_{x, D}$ can also be described in terms of Brownian motion started at $x$. If $(B_t)_{t \geq 0}$ is a Brownian motion started at $x$, and $\tau$ is defined to be the first exit time from the domain $D$, i.e. 
$$
\tau = \inf\{ \ t > 0 \ | \ B(t) \notin D \ \},
$$
then we have
$$
\int_{\partial D} f \ d\mu_{x, D} = \mathbb{E}(f(B_{\tau}))
$$
for all continuous functions $f$ on $\partial D$. 

\medskip

The harmonic measures also allow us to define the {\it Poisson integral} of any $L^{\infty}$ function $f \in L^{\infty}(\partial D)$, which is a bounded harmonic function on $D$ defined by
$$
P[f](x) := \int_{\partial D} f \ d\mu_{x, D} \ , \ x \in D.
$$
 
\medskip

For a complete, simply connected Riemannian manifold $X$ of pinched negative curvature, i.e. with sectional curvature $K$ satisfying $-b^2 \leq K_X \leq -a^2$ for some positive constants $b \geq a > 0$, it is well-known that the Dirichlet problem at infinity is solvable (\cite{AS}, \cite{Sullivan}), and so in this case one can also define a family of harmonic measures $\{ \mu_x \}_{x \in X}$, which are probability measures on $\partial X$ defined by
$$
\int_{\partial X} f \ d\mu_{x} = u_f(x)
$$
for any continuous function $f$ on $\partial X$, where $u_f$ is the solution of the Dirichlet problem at infinity with boundary value $f$, i.e. $u_f$ is harmonic on $X$ and $u_f(x) \to f(\xi)$ as $x \to \xi \in \partial X$, for any $\xi \in \partial X$. As in the case of a bounded domain, the harmonic measures $\mu_x$ are mutually absolutely continuous. As before, in this case also the harmonic measure $\mu_x$ can be described in terms of Brownian motion started at $x$. If $(B(t))_{t \geq 0}$ is a Brownian motion started at $x$, then it is known (\cite{Sullivan}) that almost every sample path of the Brownian motion converges to a (random) point $B_{\infty}$ in $\partial X$. This limiting point $B_{\infty}$ is a random variable taking values in $\partial X$, whose distribution is precisely the harmonic measure $\mu_x$; for any continuous function $f$ on $\partial X$, we have
$$
\int_{\partial D} f \ d\mu_{x, D} = \mathbb{E}(f(B_{\infty})).
$$
As in the case of a precompact domain, in this case too we can define the Poisson integral of any $L^{\infty}$ function on $\partial X$, where the measure class on $\partial X$ is that defined by the harmonic measures. As before, the Poisson integral of a function $f \in L^{\infty}(\partial X)$ is a bounded harmonic function on $X$ defined by
$$
P[f](x) := \int_{\partial X} f \ d\mu_{x} \ , \ x \in X.
$$

\medskip

Finally, it is also possible to define harmonic measures on the boundaries of horoballs. While any horoball $H = H(\xi, r)$ is not precompact in $X$, it is nevertheless precompact in $\overline{X}$, and its boundary in $\overline{X}$ is given by the compactified horosphere $S := \partial H \cup \{\xi\}$, which is a compact subset of $\overline{X}$. The Dirichlet problem for the horoball $H$ can then be stated as follows: given a continuous function $f \in C(S)$, find a continuous extension $u \in C(H \cup S)$ such that $u$ is harmonic in $H$. It is well-known from the Perron method that the Dirichlet problem is solvable if and only if there is a superharmonic barrier at each point of the boundary $S = \partial H \cup \{\xi\}$. For points $z \in \partial H$, the existence of a barrier follows from the fact that $\partial H$ is $C^2$, while for the boundary point $\xi \in \partial X$ of the horoball, one may take as a barrier a function of the form $\phi(x) = \exp(-\delta d(o,x)), x \in X$, for any constant $\delta$ 
such that $0 < \delta < (n-1)a$. Such a function is known to be superharmonic on $X$, extends continuously to $H \cup S$, and has the correct boundary behaviour at $\xi$, namely $\phi(\xi) = 0$ and $\phi(x) > 0$ for all $x \in (H \cup S) \setminus \{\xi\}$. 

\medskip

Thus the Dirichlet problem is solvable for the horoball $H$, and then as before the solution of the Dirichlet problem gives rise to a family of harmonic measures $\mu_{x, H}$ on the boundary $S = \partial H \cup \{\xi\}$ of the horoball in $\overline{X}$. 

\medskip

\section{Concentration of Harmonic measures and convergence of Poisson integrals}
\subsection{Concentration of Harmonic measures}
In this subsection we prove the results on concentration of measures mentioned in the Introduction. We show that the harmonic measures of the complement of a neighborhood of a boundary point of a limiting domain go to zero as the basepoints in a sequence of domains converge to the boundary point. The following is our first result in this direction and uses the concept of local superharmonic barriers.

\medskip

Throughout $M$ denotes a Riemannian manifold.

\medskip
 
\begin{lemma} 
\label{har_meas_domains}
Let $D \subset M$ be a domain with $C^2$ boundary. 
Let $\xi \in \partial D$, and let $U$ be a precompact neighborhood of $\xi$ such that $\overline{D} \cap U$ is $C^2$-diffeomorphic to a closed half-ball in $\R^m$. Let $\{D_n\}_{n=1}^\infty$ be a sequence of precompact domains in $M$ with 
$C^2$ boundary, and let $\{g_n\}_{n=1}^\infty$ be a sequence of $C^2$ local diffeomorphisms from $U$ into $M$, with $U_n:=g_n(U)$ satisfying
\begin{itemize}
\item[(i)] $g_n\left(U \cap \partial D\right) = U_n \cap \partial D_n$ and $g_n\left(U \cap  D\right) = U_n \cap D_n$\:, for all $n \in \N$,
\item[(ii)] $g_n \to id$ in $C^2$ norm on $U$.
\end{itemize}

Then for $y_n \in D_n$ such that $y_n \to \xi$ as $n \to \infty$, we have,
\begin{equation*}
\mu_{y_n,D_n}(\partial D_n \setminus U) \to 0 \text{ as } n \to \infty \:.
\end{equation*} 
\end{lemma}
\begin{proof}
Let $W_1,W_2$ be two neighborhoods of $\xi$ such that $\overline{W_2} \subset W_1$ and $\overline{W_1} \subset U$. Let $\varphi \in C^0(M)$ such that $0 \le \varphi \le 1 \:, \:\varphi \equiv 0$ on $W_1$ and $\varphi \equiv 1$ on $M \setminus U$. Then
\begin{equation} 
\label{har_meas_domains_eq1}
\mu_{y_n,D_n}(\partial D_n \setminus U) \le \int_{\partial D_n} \varphi(\zeta) \:d\mu_{y_n,D_n}(\zeta) = u_n(y_n) \:,
\end{equation}
where $u_n$ is the harmonic extension of $\varphi|_{\partial D_n}$ to $D_n$.

Now by the regularity of  $\overline{D} \cap U$, there exists a $C^2$ diffeomorphism
\begin{equation*}
\psi : U  \to V \subset \R^m \:,
\end{equation*}
such that $V$ is an open neighborhood of $0$\:,\: $\psi(\xi)=0\:,$
\begin{eqnarray*}
&\psi(\partial D \cap U)&= \{x \in \R^m : \|x+e_1\|=1\} \cap V \:, (\text{where } e_1=(1,0,\dots,0)) \\
&\psi(\overline{D} \cap U \setminus \{\xi\})& \subset  \{x \in \R^m : \|x+e_1\|\le1\} \cap V \setminus \{0\} \:. 
\end{eqnarray*} 

Let $\xi_n:= g_n(\xi)\in U_n \cap \partial D_n$ for all $n \in \N$. Then consider the $C^2$ diffeomorphisms 
\begin{equation*}
\psi_n : U_n \to V \subset \R^m \:,
\end{equation*}
defined by,
\begin{equation} \label{har_meas_domains_eq2}
\psi_n = \psi \circ g^{-1}_n \:.
\end{equation}

Then we note that  
\begin{eqnarray*}
&\psi_n(\xi_n)&=0\:, \\
&\psi_n(\partial D_n \cap U_n)&= \{x \in \R^m : \|x+e_1\|=1\} \cap V \:, (\text{where } e_1=(1,0,\dots,0)) \\
&\psi_n(\overline{D_n} \cap U_n \setminus \{\xi_n\})& \subset  \{x \in \R^m : \|x+e_1\|\le1\} \cap V \setminus \{0\} \:. 
\end{eqnarray*}

Now as $g_n \to id$ in $C^2$ norm on $U$, we have $\xi_n \to \xi$. Moreover (\ref{har_meas_domains_eq2}) implies that $\psi_n \to \psi$ in  $C^2$ norm and hence, in particular $\left\{{\|\psi_n\|}_{C^2(\overline{D_n} \cap U_n)}\right\}_{n=1}^\infty$ is uniformly bounded.

For $h \in C^2(V)$\:, consider for $n \in \N$ , $L_nh:=\left(\Delta(h \circ \psi_n)\right) \circ \psi^{-1}_n$\:, which is a second order elliptic operator,
\begin{equation*}
L_n = \sum_{i,j} a_{ij,n}\: \frac{\partial^2}{\partial x_i \partial x_j} \:+\: \sum_{i} b_{i,n}\: \frac{\partial}{\partial x_i} \:. 
\end{equation*}

By uniform boundedness of $C^2$ norms of $\psi_n$, we get that for all $n \in \N$, there exists $\lambda_1\:,\:\lambda_2 \:>\:0$ such that
\begin{eqnarray*}
&&\sum_{i,j} a_{ij,n}(x)\: v_iv_j \ge  \lambda_1 {\|v\|}^2 \text{ and } \\
&&|b_{i,n}(x)| \le  \lambda_2 \:, \text{ for all } i\:,\text{ for all } x \in V\cap \psi_n(\overline{D_n} \cap U_n)\:,\: v=(v_1,\dots,v_m) \in \R^m \:.   
\end{eqnarray*}

Let $s(x):= 1 - e^{\alpha x_1} $\: for $\alpha > 0$. Then $s(0)=0$ and 
\begin{equation*}
s(\{x \in \R^m : \|x+e_1\|\le1\}  \setminus \{0\}) \subset (0, \infty) \:.
\end{equation*}

Also for all $x \in V\cap \psi_n(\overline{D_n} \cap U_n)$\:,
\begin{eqnarray*}
L_n\:s(x) &=& - (\alpha^2 \:a_{11,n} + \alpha\: b_{1,n}) e^{\alpha x_1} \\
& \le & -(\lambda_1\: \alpha^2 - \lambda_2\: \alpha )e^{\alpha x_1} \\
& \le & 0 \:,\:\text{ for } \alpha \text{ sufficiently large .}
\end{eqnarray*}

It follows that $s_n:= s \circ \psi_n$ is superharmonic on $D_n \cap U_n$, with $s_n(\xi_n)=0$, for all $n$. Then by choosing $\alpha$ sufficiently large in the definition of $s$ and replacing $s$ by $2s$, we can ensure that 
\begin{equation*}
s_n \ge 1 \:\: \text{on } (\overline{D_n} \cap U_n) \setminus W_2 \:,
\end{equation*}
for large $n$.

We note due to hypothesis $(ii)$, for large $n$, $\overline{W_2} \subset U_n$. Thus we  extend $s_n$ to all of $\overline{D_n}$ by,
\begin{equation*}
\tilde{s_n}(z) = \begin{cases} 
       \min\{1, \:s_n(z)\} & \text{ if } z \in \overline{D_n} \cap W_2 \\
       1 & \text{ if } z \in \overline{D_n} \setminus W_2 \:.
   \end{cases}
\end{equation*}

We note that $\tilde{s_n}$ is superharmonic, non-negative and for $n$ sufficiently large, $\tilde{s_n} \ge \varphi$ on $\partial D_n$. The last fact is seen as follows.

On $\partial D_n \cap W_2$,
\begin{equation*}
\varphi \equiv 0 \le \tilde{s_n} \:,
\end{equation*}
and on $\partial D_n \setminus W_2$,
\begin{equation*}
\varphi \le 1 \equiv \tilde{s_n} \:.
\end{equation*}

Hence by the maximum principle, it follows that for $n$ large,
\begin{equation} \label{har_meas_domains_eq3}
u_n(y_n) \le \tilde{s_n}(y_n) \le s_n(y_n) \:.
\end{equation}

Now as $y_n \to \xi$ and $g_n \to id$ uniformly, it follows that $g^{-1}_n(y_n) \to \xi$. Hence for $n$ large,
\begin{equation}\label{har_meas_domains_eq4}
s_n(y_n) = s(\psi_n(y_n)) = s(\psi(g^{-1}_n(y_n))) \to s(0) = 0 
\end{equation}

Hence, combining (\ref{har_meas_domains_eq1}), \:(\ref{har_meas_domains_eq3}) and (\ref{har_meas_domains_eq4}), it follows that
\begin{equation*}
\mu_{y_n,D_n}(\partial D_n \setminus U) \le  u_n(y_n) \le s_n(y_n) \to 0 \text{ as } n \to \infty \:.
\end{equation*}
\end{proof}

\medskip

Note that the above Lemma applies in particular if we have a sequence of balls 
$B_n = B(x_n, r_n)$ converging to a ball $B = B(x, r)$ (meaning that $x_n \to x$ and $r_n \to r$ as $n \to \infty$) and $y_n \in B_n$ converges to $\xi \in \partial B$, or if we have a sequence of balls $B_n = B(x_n, r_n)$ converging to a horoball $H = H(\xi, r)$ (meaning that $x_n \to \xi$ and $r_n - d(x_n,o) \to r$ 
as $n \to \infty$) and $y_n \in B_n$ converges to $z \in \partial H$.

\medskip

The next result can be viewed as a global analogue of Lemma \ref{har_meas_domains} and uses the concept of global barriers. In the statement of the Lemma below, closures of domains in $M$ are taken in the compactification $\overline{M}$.

\medskip

\begin{lemma} \label{har_meas_compl_cone}
Let  $M$ be a complete, simply connected Riemannian manifold, satisfying sectional curvature bounds $-b^2 \le K_M \le -a^2$, for some $0<a\le b$. Let $W_n$ be a sequence of domains in $M$ where $W_n$ is either a ball, a horoball or the full space $M$ such that $\overline{W_n}$ converges to $\overline{W}$ in the Hausdorff topology on compacts in $\overline{M}$, where $W$ is either a horoball or the full space $M$. 

\medskip

Let $z_n \in W_n$ be a sequence such that $z_n \to \xi$ as $n \to \infty$, for some $\xi \in \overline{W} \cap \partial M$. Then for any cone neighborhood $U$ of $\xi$,
\begin{equation*} 
\mu_{z_n,W_n}(\partial W_n \setminus U) \to 0 \: \text{ as }  n \to \infty\:.
\end{equation*}
\end{lemma}
\begin{proof}
Passing to a subsequence we may assume that either all $W_n=M$ or none of them is $M$.

In the former case, the result follows from the solution of the Dirichlet problem on $\partial M$ as the harmonic measures $\mu_{z_n}$ converge weakly to the Dirac mass at $\xi$.

In the latter case, $W_n$ is either a ball or a horoball. Depending on the limiting domain, this case has two subcases:
\begin{itemize}
\item[(i)] the limiting domain is a horoball,
\item[(ii)] the limiting domain is all of $M$. 
\end{itemize}
In case $(i)$ let $V$ be another cone neighborhood of $\xi$ such that $\overline{V} \subset U$. Let $\psi \in C^0(\overline{M})$ be such that $0 \le \psi \le 1\:, \psi \equiv 0$ on $V$ and $\psi \equiv 1$ on $\overline{M} \setminus U$. Then note that
\begin{equation} \label{har_meas_compl_cone_eq1}
\mu_{z_n,W_n}(\partial W_n \setminus U) \le \int_{\partial W_n} \psi(y) \:d\mu_{z_n,W_n}(y) = u_n(z_n) \:,
\end{equation}
where $u_n$ is the harmonic extension of $\psi|_{\partial W_n}$ to $W_n$. Now note that for $n$ large enough, $\partial W_n \setminus V$ is contained in a fixed ball $\tilde{B}$ centered at the origin. 

Fix a $\delta \in (0, (m-1)a)$, where $m$ is the dimension of $M$ and $-a^2$ is the assumed upper bound on the sectional curvature. For such a $\delta$, as shown in \cite{AS}, the function
\begin{equation*}
s(x) := \frac{1}{\varepsilon} e^{-\delta \:d(o,x)} \:, \: x \in M \:,
\end{equation*}
is superharmonic on $M$.

\medskip

Now there exists $\varepsilon > 0$ such that 
\begin{equation*}
e^{-\delta \:d(o,x)} > \varepsilon \:, \text{ for all } x \in \tilde{B} \:.
\end{equation*}

Then note that $s$ is non-negative and $s \ge 1$ on $\tilde{B}$ and hence also on $\partial W_n \setminus V$, for $n$ large enough. Thus for $n$ sufficiently large, $\psi \le s$ on $\partial W_n$. This is seen as follows.

On $\partial W_n \setminus V$
\begin{equation*}
\psi \le 1 \le s \:,
\end{equation*} 
and on $\partial W_n \cap V$
\begin{equation*}
\psi \equiv 0 \le s\:.
\end{equation*} 

Hence by the maximum principle, it follows that for $n$ large enough,
\begin{equation} \label{har_meas_compl_cone_eq2}
u_n(z_n) \le s(z_n) \:. 
\end{equation}

Then combining (\ref{har_meas_compl_cone_eq1}) and (\ref{har_meas_compl_cone_eq2}), it follows that
\begin{equation*}
\mu_{z_n,W_n}(\partial W_n \setminus U) \le u_n(z_n) \le s(z_n) = \frac{1}{\varepsilon} e^{-\delta \:d(o,z_n)} \:\to 0 \text{ as } n \to \infty\:.
\end{equation*}

This completes the proof for the case $(i)$.

\medskip

Now for case $(ii)$ we start by noting that the Dirichlet problem is solvable for a cone neighborhood $U$ such that $\partial U \cap M$ is $C^2$ (we may assume without loss of generality that $U$ satisfies this hypothesis). This is seen as follows. By the well-known Perron method for the solution of the Dirichlet problem, it is enough to show that there exists a barrier at each point of the boundary $\partial U \subset \overline{M}$. For each point $z \in \partial U \cap M$ on the finite part of the boundary, there exists a barrier at $z$ by the $C^2$ regularity of $\partial U \cap M$. For a point $\eta \in \partial U \cap \partial M$ on the part of the boundary of $U$ in $\partial M$, we can use the fact that the Dirichlet problem is solvable for $M$: for a barrier at $\xi$ we take the restriction to $\overline{U}$ of the harmonic extension of a function $\psi \in C^0(\partial M)$ such that $\psi(\eta)=0$ and $\psi(\zeta)>0$ for all $\zeta \in \partial M - \{\eta\}$.

\medskip

Then as the domains $W_n$ converge to all of $M$ and $z_n \to \xi$, we have for large $n$, $z_n \in U$ but $\overline{W_n} \setminus U$ is non-empty. Note that for any Brownian path which starts at $z_n$ and first exits the domain $W_n$ through $\partial W_n \setminus U$, the first exit point of this path from $U$ must lie in the finite part $\partial U \cap M$ of $\partial U$. The probabilistic interpretation of harmonic measures then implies that
\begin{equation} \label{har_meas_compl_cone_eq3}
\mu_{z_n,W_n}\left(\partial W_n \setminus U\right) \le \mu_{z_n,U}\left(\partial U \cap M\right) \:.
\end{equation}
Now the solvability of the Dirichlet problem for $U$ noted earlier implies that
harmonic measures on $U$ with respect to $z_n$ converge weakly to the Dirac mass at $\xi$. This implies
\begin{equation} \label{har_meas_compl_cone_eq4}
\mu_{z_n,U}\left(\partial U \cap M\right) \to 0\:,\text{ as } n \to \infty \:.
\end{equation} 
from which it follows that
\begin{equation*}
\mu_{z_n,W_n}\left(\partial W_n \setminus U\right) \to 0 \:,\text{ as } n \to \infty \:.
\end{equation*}
This completes the proof of the Lemma.
\end{proof}
\subsection{Convergence of Poisson integrals}
In this subsection we show that convergence of domains implies 
convergence of the corresponding Poisson integrals of a function having nice boundary behaviour. We first prove some preliminary results.

\medskip

\begin{lemma} \label{zero_meas_set}
Let $B$ and $D$ be pre-compact domains with smooth boundary such that $B \subset D$. Let $A \subset \partial B \cap \partial D$ be such that $\mu_{x,D}(A)=0$ for all $x \in D$, then $\mu_{x,B}(A)=0$ for all $x \in B$.
\end{lemma} 

\medskip

\begin{proof}
Let $x \in B$. If the Brownian motion $(B_t)_{t \ge 0}$ starting at $x$ first exits the domain $B$ via the boundary portion $A$, then in particular it also first exits the domain $D$ via the boundary portion $A$. If $\tau_1$ and $\tau_2$ are the first exit times for the Brownian motion from the domains $B$ and $D$ respectively, then we have
\begin{equation*}
\mu_{x,B}(A) = Pr\{B_{\tau_1} \in A\} \le Pr \{B_{\tau_2} \in A\} = \mu_{x,D}(A)=0 \:.
\end{equation*}  
\end{proof}

\begin{lemma} \label{no_atom}
Let  $M$ be a complete, simply connected Riemannian manifold of dimension $\ge 2$, satisfying sectional curvature bounds $-b^2 \le K_M \le -a^2$, for some $0<a\le b$. Let $H$ be a horoball in $M$ based at $\xi \in \partial M$. Let $\tau$ denote the first exit time of Brownian motion from $H$ Then:

\medskip

\noindent (1) $\tau$ is finite almost surely.

\medskip

\noindent (2) $\{\xi\}$ is not an atom for the harmonic measures on the compactified horosphere $\partial H \cup \{\xi\}$.
\end{lemma}
\begin{proof}
(1) Choose and fix an interior point $z \in H$. By taking limits of radii in Theorem $1.1$ of \cite{BH} we get that the harmonic measure on $\partial X$ with respect to $z$ has no atoms. By the probabilistic interpretation of the harmonic measures, it follows that for Brownian motion starting from $z$, the limiting point $B_{\infty}$ lies in $\partial X - \{\xi\}$ with probability $1$. Since $\overline{H} \cap \partial M = \{\xi\}$, this implies that with probability $1$  the Brownian motion starting from $z$ first exits $H$ at some finite boundary point of $H$ (that is points lying on $\partial H$). Thus $\tau < \infty$ almost surely. 

\medskip

(2) The above means that
\begin{equation*}
\mu_{z,H}\left(\partial H \right)=1 \:,
\end{equation*} 
which implies
\begin{equation*}
\mu_{z,H}\left(\{\xi\}\right)=0 \:.
\end{equation*}
As $z$ was an arbitrary interior point, the result follows.
\end{proof}

\medskip

\begin{lemma} \label{conv_stopping_time}
Let $D$ be either of the following:

\medskip

\noindent (1) A precompact domain in a Riemannian manifold $M$, or 

\medskip

\noindent (2) $D = M$, a complete, simply connected Riemannian manifold of pinched negative curvature $-b^2 \le K_M \le -a^2$, for some $0<a \le b$. In this case $\overline{D}$ denotes the compactification $\overline{M}$.

\medskip

Let $W_n \subset D$ be a sequence of domains such that $\overline{W_n} \to \overline{W}$ and $\partial W_n \to \partial W$ in the Hausdorff topology on compacts in $\overline{D}$, for some domain $W \subset D$.

\medskip

Let $\gamma : [0, \infty) \to M$ be a continuous path starting from some $x \in W$. Let $\tau_n, \tau$ be the first exit times of the path $\gamma$ from $W_n,W$ respectively. 

\medskip

Then $\tau_n \to \tau$ as $n \to \infty$. 
\end{lemma}
\begin{proof}

Note that the hypothesis of the Lemma implies that $x \in W_n$ for all $n$ large enough, we only consider such $n$. We consider two cases:

\medskip

\noindent (1) $\tau < +\infty$: 

\medskip
 
In this case, given $\epsilon > 0$, by the definition of the first exit time there exists $t>0$ with $\tau < t < \tau(\gamma) + \varepsilon$ such that $\gamma(t)$ lies outside an open neighbourhood $U$ of $\overline{W}$. Then by the Hausdorff convergence hypothesis, we have $\overline{W_n} \subset U$ for all $n$ large enough, thus $\gamma(t)$ lies outside of $\overline{W_n}$, and so $\tau_n \leq t$. Thus for $n$ large,
\begin{equation*}
\tau_n \le t < \tau + \varepsilon \:.
\end{equation*}

\medskip

On the other hand by definition of the first exit time, the portion of the path $\gamma([0, \tau - \varepsilon])$ is contained in $W - V$ for some open neighbourhood $V$ of $\partial W$. By the Hausdorff convergence of $\partial W_n$ to $\partial W$, for $n$ large enough we have $\partial W_n \subset V$, thus $\gamma([0, \tau - \varepsilon])$ is contained in $W_n$, hence 
\begin{equation*}
\tau_n > \tau - \epsilon \:.
\end{equation*}
This completes the proof in the case $\tau < +\infty$.

\medskip

\noindent (2) $\tau = +\infty$:

\medskip

In this case, given $T > 0$, since $\tau = +\infty$ there exists a neighbourhood $V$ of $\partial W$ such that $\gamma([0,T]) \subset W - V$. Now as before, the Hausdorff convergence implies that for all $n$ large enough we have $\partial W_n \subset V$, hence $\gamma([0,T]) \subset W_n$ and so $\tau_n \geq T$. It follows that $\tau_n \to +\infty = \tau$ as $n \to \infty$. 

%
%
\end{proof}

\begin{lemma}
\label{poisson_conv_1}
Let $D$ be either of the following:

\medskip

\noindent (1) A precompact domain in a Riemannian manifold $M$, or 

\medskip

\noindent (2) $D = M$, a complete, simply connected Riemannian manifold of pinched negative curvature $-b^2 \le K_M \le -a^2$, for some $0<a \le b$. In this case $\overline{D}$ denotes the compactification $\overline{M}$.

\medskip

Let $W_n \subset D$ be a sequence of domains such that
\begin{itemize}
\item $\partial W_n \cap \partial D = \emptyset$ \:,
\item $\overline{W_n} \to \overline{W}$ and $\partial W_n \to \partial W$ in the Hausdorff topology on compacts in $\overline{D}$, for some domain $W \subset D$.
\end{itemize}

\medskip

Moreover, in case $(1)$ we assume that $W_n$ and $W$ are balls, while in case $(2)$ we assume that $W_n$ are balls and $W$ is either a ball, a horoball or all of $M$.

\medskip

Let $u$ be bounded, continuous on $D$ and satisfy 
\begin{equation} \label{poisson_conv_1_bdry}
\displaystyle\lim_{\substack{z \to \xi\\ z \in D}} u(z) = f(\xi) \:\text{exists for almost every }\xi \in \partial D \:,
\end{equation}
with respect to the harmonic measures on $\partial D$, where $f$ is an $L^{\infty}$ function on $\partial D$. 

\medskip

Then the Poisson integrals of $u$ on $W_n$ converge pointwise on $W$ to the Poisson integral of $u$ on $W$.  
\end{lemma}
\begin{proof}
First let us assume that $D$ is as in $(1)$ that is a bounded, pre-compact domain. Then $W$ is either of the following:
\begin{itemize}
\item[(i)] $W$ is a ball in $D$ such that $\partial W \cap \partial D$ has zero harmonic measure,
\item[(ii)] $W$ is a ball in $D$ such that $\partial W \cap \partial D$ has positive harmonic measure.
\end{itemize}

\medskip

In both the cases $(i)$ and $(ii)$, by continuity, Lemma \ref{zero_meas_set} and boundary behaviour (\ref{poisson_conv_1_bdry}), it follows that boundary limits of $u$ exist for almost all points on $\partial W$ (with respect to the Riemannian measure of $\partial W$) and hence $u$ defines an $L^\infty$ function on $\partial W$. Hence the Poisson integral of $u$ on $W$ is well-defined.

\medskip

We choose and fix $x \in W$. By convergence of domains, $x \in W_n$ for all $n$ sufficiently large. We consider the Brownian motion ${(B_t)}_{t \ge 0}$ starting at $x$. By convergence of first exit times in Lemma \ref{conv_stopping_time}, we have
\begin{equation*}
B_{\tau_n} \to B_\tau \: \text{as } n \to \infty \:,
\end{equation*} 
on almost all Brownian paths, where $\tau_n,\tau$ are the first exit times of the Brownian motion (starting at $x$) from $W_n,W$ respectively. Let $A$ be the full harmonic measure set in $\partial W$ where boundary limits of $u$ exist. Now by the probabilistic interpretation of harmonic measures, almost all Brownian paths starting at $x$ land in $A$. Consequently it follows that
\begin{equation*}
u(B_{\tau_n}) \to u(B_\tau) \: \text{ almost surely as } n \to \infty \:.
\end{equation*}

\medskip

Then by boundedness of $u$, the Dominated Convergence Theorem is applicable on the probability space of Brownian paths starting at $x$ and it yields,
\begin{equation*}
\E(u(B_{\tau_n})) \to \E(u(B_{\tau})) \: \text{as } n \to \infty \:,
\end{equation*}
that is,
\begin{equation*}
\int_{\partial W_n} u(\xi) \: d\mu_{x,W_n}(\xi) \to \int_{\partial W} u(\xi) \: d\mu_{x,W}(\xi) \:,
\end{equation*}
which gives the result in this case.

\medskip

Now for the case when $D$ is as in $(2)$, $W$ is one of the following:
\begin{itemize}
\item[(i)] a ball,
\item[(ii)] a horoball,
\item[(iii)] all of $M$.
\end{itemize} 

\medskip

In case (ii) it follows from Lemma \ref{no_atom} that $u$ is defined almost everywhere with respect to the harmonic measure on $\partial W$, and so the Poisson integral of $u$ on $W$ is well-defined, while in case (iii) we take the Poisson integral of $u$ to be the Poisson integral of the boundary values on $\partial M$ given by (\ref{poisson_conv_1_bdry}), which are defined almost everywhere with respect to harmonic measure on $\partial M$. Thus in all cases the Poisson integral of $u$ on $W$ is well-defined.

\medskip 

Now for cases $(i)$ and $(ii)$, by Lemma \ref{no_atom}, the first exit times are finite almost surely and hence arguments similar to that in the case of bounded domains yield the result.

\medskip

Finally, for case $(iii)$, we choose and fix $x \in M$. Then $x \in W_n$, for $n$ sufficiently large. By Lemma \ref{conv_stopping_time}, we have $\tau_n \to \infty$ almost surely as $n \to \infty$. Recalling from Sullivan's result \cite{Sullivan} that the Brownian motion $B_t$ converges almost surely as $t$ tends to infinity to a random variable $B_{\infty}$ taking values in $\partial M$, we have
\begin{equation*}
B_{\tau_n} \to B_\infty\:\text{ almost surely as } n \to \infty. \:
\end{equation*}

\medskip

Now the boundary limits of $u$ exist for all points in a subset $A \subset \partial M$ of full harmonic measure. Since the harmonic measure on $\partial M$ is the distribution of $B_{\infty}$, we have $B_{\infty} \in A$ almost surely. It follows that
\begin{equation*}
u(B_{\tau_n}) \to f(B_{\infty}) \:\text{ almost surely as } n \to \infty \:
\end{equation*}

\medskip

Then by the Dominated Convergence Theorem on the probability space of paths, we get
\begin{equation*}
\int_{\partial W_n} u(\xi) d\mu_{x,W_n}(\xi) = \E(u(B_{\tau_n})) \to \E(f(B_\infty))= \int_{\partial M} f(\xi) d\mu_x(\xi) \:.
\end{equation*}
\end{proof}

\medskip

We now apply Lemma \ref{poisson_conv_1} to obtain the following more general version of convergence of Poisson integrals:

\medskip

\begin{lemma}
\label{poisson_conv_2}
Let $D$ be either of the following:

\medskip

\noindent (1) A precompact domain in a Riemannian manifold $M$, or 

\medskip

\noindent (2) $D = M$, a complete, simply connected Riemannian manifold of pinched negative curvature $-b^2 \le K_M \le -a^2$, for some $0<a \le b$. In this case $\overline{D}$ denotes the compactification $\overline{M}$.

\medskip

Let $W_n \subset D$ be a sequence of domains such that $\overline{W_n} \to \overline{W}$ and $\partial W_n \to \partial W$ in the Hausdorff topology on compacts in $\overline{D}$, for some domain $W \subset D$.

\medskip

Moreover we assume that in case $(1)$ $W_n$ and $W$ are balls, while in case $(2)$ $W_n$  and $W$ are either balls, horoballs or all of $M$.

\medskip

Let $u$ be bounded, continuous on $D$ and satisfy
\begin{equation*} 
\displaystyle\lim_{\substack{z \to \xi\\ z \in D}} u(z) = f(\xi) \:\text{exists for almost every }\xi \in \partial D \:,
\end{equation*}
with respect to the harmonic measures on $\partial D$. 

\medskip

Then the Poisson integrals of $u$ on $W_n$ converge pointwise on $W$ to the Poisson integral of $u$ on $W$. 

\medskip

In fact the convergence is uniform on compacts in $W$. 
\end{lemma}
\begin{proof}
We choose and fix $x \in W$. By convergence of domains, $x \in W_n$ for all $n$ sufficiently large. For any such $n$, we can choose a ball $V_n \subset D$ such that 
\begin{itemize}
\item $\partial V_n \cap \partial D = \emptyset$\:,
\item $x \in V_n$\:,
\item Hausdorff distances between the closures $\overline{V_n}$ and $\overline{W_n}$, as well as the boundaries $\partial V_n$ and $\partial W_n$ are sufficiently small so that by Lemma \ref{poisson_conv_1}, we have
\begin{equation} \label{poisson_conv_2_eq1}
\left|P_{V_n}[u](x)-P_{W_n}[u](x)\right| < \frac{1}{n} \:,
\end{equation}
where $P_{V_n}[u]$ and $P_{W_n}[u]$ are the Poisson integrals of $u$ on $V_n$ and $W_n$ respectively, and
\item $\overline{V_n} \to \overline{W}$ and $\partial V_n \to \partial W$ in the Hausdorff topology on compacts in $\overline{D}$.
\end{itemize}

\medskip

Then Lemma \ref{poisson_conv_1} implies that given $\varepsilon >0$, for $n$ sufficiently large, we have
\begin{equation} \label{poisson_conv_2_eq2}
\left|P_{V_n}[u](x)-P_{W}[u](x)\right| < \frac{\varepsilon}{2} \:,
\end{equation}
where $P_{V_n}[u]$ and $P_W [u]$ are the Poisson integrals of $u$ on $V_n$ and $W$ respectively. Then by (\ref{poisson_conv_2_eq1}) and (\ref{poisson_conv_2_eq2}), for $n$ large we have
\begin{equation*} 
\left|P_{W_n}[u](x)-P_{W}[u](x)\right| < \varepsilon \:.
\end{equation*}

\medskip

This proves the pointwise convergence of the Poisson integrals. The uniform convergence on compacts of the Poisson integrals follows from this pointwise convergence and from the fact that the Poisson integrals are uniformly bounded harmonic functions. The uniform boundedness implies by classical results that the Poisson integrals are equicontinuous on compacts in $W$, and the pointwise convergence implies the uniqueness of the limit of any subsequence which converges uniformly on compacts, hence the whole sequence converges uniformly on compacts.
\end{proof}

\section{Construction of an auxiliary subharmonic function}
For $z \in D$, under the hypothesis of Theorem \ref{domain_bdry_bhv}, we define $\rho_0(z)=d(z, \partial D)$ and let $\mathscr{F}_z$ be the collection of 
 harmonic extensions of $u$ on balls $B(x,r)$ where $x \in D, r \le \rho_0(x), z \in B(x,r)$ and we have 
\begin{equation*}
u(x)=\int_{\partial B(x,r)} u(y)\:d\mu_{x,B(x,r)}(y) \:
\end{equation*}
(in the case where $r = \rho_0(x)$, we note that by Lemma \ref{zero_meas_set}, 
$u$ is defined almost everywhere on $\partial B(x,r)$ with respect to the harmonic measure on $\partial B(x,r)$, and so by the harmonic extension of $u$ on $B(x,r)$ we mean the Poisson integral of $u$ on $B(x,r)$).

\medskip

For $z \in D$, under the hypothesis of Theorem \ref{domain_radius}, define $\rho_0$ by,
\begin{equation} \label{rho_0}
\rho_0(z) = \begin{cases} 
       \kappa \: d(z,\partial D) & \text{ if } d(z,\partial D) \ge (1-\tau)r_D \\
      \min\{\kappa \: d(z,\partial D) \:,\:\delta(r(z),\varepsilon(d(z,\partial D));u)\} & \text{ if } d(z,\partial D) < (1-\tau)r_D \:,
   \end{cases}
\end{equation}
and let $\mathscr{F}_z$ be the collection of harmonic extensions of $u$ on balls $B(x,r)$ where $x \in D, r \le \rho_0(x), z \in B(x,r)$ and 
\begin{equation*}
u(x)=\int_{\partial B(x,r)} u(y)\:d\mu_{x,B(x,r)}(y) \:.
\end{equation*}

\medskip

Finally for $z \in M$, under the hypothesis of Theorem \ref{neg_bdry_bhv}, we define $\mathscr{F}_z$ to be the collection of
\begin{itemize}
\item harmonic extensions of $u$ on balls $B(x,r) \subset M$ where $z \in B(x,r)$ and 
\begin{equation*}
u(x)=\int_{\partial B(x,r)} u(y)\:d\mu_{x,B(x,r)}(y) \:,
\end{equation*}
\item harmonic extensions of $u$ (i.e. the Poisson integral of the $L^\infty$ function $u$ on the horospheres, which is well-defined by Lemma \ref{no_atom}) on horoballs that contain $z$,
\item the harmonic extension of $u$ (i.e. the Poisson integral of the $L^\infty$ function $u$ on $\partial M$) on $M$ \:.
\end{itemize}

\medskip

In the rest of this section $D$ will be used to denote both a bounded, pre-compact domain (as in Theorems \ref{domain_bdry_bhv}, \ref{domain_radius}) as well as a negatively curved Hadamard manifold (as in Theorem \ref{neg_bdry_bhv}) according to context, with more details specified when required. 

\medskip

We now define,
\begin{equation} \label{defn_v}
v(z):= \sup_{h \in \mathscr{F}_z} h(z) \:, \text{ for all } z \in D.
\end{equation}

\medskip

Then note that, since $u$ satisfies the RMVP, for any $z \in D$ there exists $h \in \mathcal{F}_z$ such that $u(z) = h(z)$, and hence  
\begin{equation}
u(z) \le v(z) \:, \text{ for all } z \in D.
\end{equation}
In the rest of this section we will prove some important properties of $v$.
\begin{lemma} \label{sup_attained}
For all $z \in D$, there exists $h \in \mathscr{F}_z$ such that $v(z)=h(z)$. 
\end{lemma}
\begin{proof}
Choose and fix $z_0 \in D$. Let $\{h_n\}_{n=1}^\infty \in \mathscr{F}_{z_0}$ such that 
\begin{equation*}
h_n(z_0) \to v(z_0) \: \text{ as } n \to \infty \:.
\end{equation*}
For each $n \in \N$, let $W_n$ be the domain such that $h_n$ is the harmonic extension of $u$ on $W_n$ (i.e. the Poisson integrals of the $L^\infty$ function $u$ on $\partial W_n$).

\medskip

Passing to a subsequence, we may assume that $\overline{W_n} \to \overline{W}$ and $\partial W_n \to \partial W$ (for some domain $W$ such that $W \subset D$) in the Hausdorff topology on compacts in $\overline{D}$. As $z_0 \in W_n$ for all $n \in \N$, there are three cases: 
\begin{enumerate}
\item[(i)] $W=\{z_0\}$\:,
\item[(ii)] $W$ is non-degenerate and $z_0$ lies in the interior of $W$,
\item[(iii)] $W$ is non-degenerate and $z_0 \in \partial W \cap D$.
\end{enumerate}

\medskip

In case $(i)$, by continuity,
\begin{eqnarray*}
|h_n(z_0) - u(z_0)| &\le & \int_{\partial W_n} |u(y)-u(z_0)|\: d\mu_{z_0,W_n}(y) \\
& \le & \sup_{y \in \partial W_n}  |u(y)-u(z_0)| \:\: \to 0 \:\text{ as } n \to \infty\:.
\end{eqnarray*} 
Hence, 
\begin{equation*}
v(z_0)=u(z_0)=h(z_0) \:,
\end{equation*}
where $h$ is the harmonic extension of $u$ on the ball $B(z_0,\rho(z_0))$\:.

\medskip

For case $(ii)$, in the setting of Theorems \ref{domain_bdry_bhv} and \ref{neg_bdry_bhv}, the hypothesis of Lemma \ref{poisson_conv_2} are satisfied. The same is also true in the setting of Theorem \ref{domain_radius}, as by the condition (\ref{rho_0}), all the domains $W_n, W$ are contained in an open domain, say $B$ such that $\overline{B} \subset D$. Then since $u$ is continuous in $\overline{B}$, the hypotheses of Lemma \ref{poisson_conv_2} applied to the function $u$ on the domain $B$ are satisfied.

\medskip 

So now applying Lemma \ref{poisson_conv_2}, we get
\begin{equation*}
h_n(z_0) \to h(z_0)\:,
\end{equation*}
where $h$ is the harmonic extension of $u$ on the limiting domain $W$ (i.e. the Poisson integral of the $L^\infty$ function $u$ on $\partial W$).

\medskip

Finally for case $(iii)$, we note that in the setting of Theorems \ref{domain_bdry_bhv}, \ref{domain_radius}, $W$ is a ball in $D$ and in the setting of Theorem \ref{neg_bdry_bhv}, $W$ is either a ball or a horoball. Then by continuity, for a given $\varepsilon>0$, there exists $\delta >0$ such that 
\begin{equation*}
|u(y)-u(z_0)| < \varepsilon/2\:, \text{ for all }y \in B(z_0, \delta) \subset D.
\end{equation*}
In the setting of Theorems \ref{domain_bdry_bhv}, \ref{neg_bdry_bhv}, set $C:= \displaystyle\sup_{D}|u|$ and in the setting of Theorem \ref{domain_radius}, set $C:= \displaystyle\sup_{B}|u|$, where $B$ is a domain containing all $W_n$ and $W$ such that $\overline{B} \subset D$. Then by Lemma \ref{har_meas_domains}, we get for $n$ large,
\begin{eqnarray*}
|h_n(z_0) - u(z_0)| &\le & \int_{\partial W_n \cap B(z_0, \delta)} |u(y)-u(z_0)| \:d\mu_{z_0,W_n}(y) \\
& + & \int_{\partial W_n \setminus B(z_0, \delta)} |u(y)-u(z_0)| \:d\mu_{z_0,W_n}(y) \\
& < & \frac{\varepsilon}{2} + 2C\: \mu_{z_0,W_n}(\partial W_n \setminus B(z_0, \delta)) \\
& < & \varepsilon \:.
\end{eqnarray*} 
Hence, 
\begin{equation*}
v(z_0)=u(z_0)=h(z_0) \:,
\end{equation*}
where $h$ is the harmonic extension of $u$ on the ball $B(z_0,\rho(z_0))$\:.
\end{proof}

\begin{lemma}
\label{cont_and_subh}
$v$ is continuous and subharmonic in $D$.
\end{lemma}
\begin{proof}
First we show that $v$ is continuous. Choose and fix $z_0 \in D$. By Lemma \ref{sup_attained} there exists $h \in \mathscr{F}_{z_0}$ such that $v(z_0)=h(z_0)$. Let $W$ be the domain such that $h$ is the harmonic extension of $u$ on $W$ (i.e. the Poisson integral of the $L^\infty$ function $u$ on $\partial W$). Note that $h \in \mathcal{F}_z$ for all $z \in W$. Then by definition of $v$ (see (\ref{defn_v})), we have
\begin{equation} \label{v_bigger_h}
v(z) \ge h(z) \ \hbox{ for all } z \in W.
\end{equation}

\medskip

Hence by continuity of $h$,
\begin{equation} \label{liminf}
\liminf_{z \to z_0} v(z) \ge \liminf_{z \to z_0} h(z) =  h(z_0)=v(z_0) \:.
\end{equation} 

\medskip

Now for the limsup, we consider $\{z_n\}_{n=1}^\infty \subset D$ such that $z_n \to z_0$ as $n \to \infty$. Let $h_n \in \mathscr{F}_{z_n}$ be the harmonic extensions of $u$ on the domains $W_n$ (i.e. the Poisson integrals of the $L^\infty$ function $u$ on $\partial W_n$) such that $v(z_n)=h_n(z_n)$. Passing to a subsequence, we may assume that $\overline{W_n} \to \overline{W}$ and $\partial W_n \to \partial W$ (for some domain $W$ such that $W \subset D$) in the Hausdorff topology on compacts in $\overline{D}$. Again as in the proof of Lemma \ref{sup_attained}, we consider three cases. 

\medskip

If $W=\{z_0\}$, then as in the proof of Lemma \ref{sup_attained}, continuity of $u$ gives
\begin{equation*}
v(z_n) = h_n(z_n) \to u(z_0) \le v(z_0) \:.
\end{equation*}

\medskip

If $W$ is non-degenerate and $z_0$ lies in the interior of $W$, then for the harmonic extension of $u$ on $W$, say $h$ (i.e. the Poisson integral of the $L^\infty$ function $u$ on $\partial W$), we get, by Lemma \ref{poisson_conv_2},
\begin{equation*}
v(z_n) = h_n(z_n) \to h(z_0) \le v(z_0) \:.
\end{equation*}

\medskip

In the third case, when $W$ is non-degenerate and $z_0 \in \partial W \cap D$, then the same argument as in the proof of Lemma \ref{sup_attained} using concentration of measure gives
\begin{equation*}
v(z_n) = h_n(z_n) \to u(z_0) \le v(z_0) \:.
\end{equation*}
Combining all the cases, we get
\begin{equation} \label{limsup}
\limsup_{z \to z_0} v(z) \le v(z_0) \:.
\end{equation}

\medskip

From (\ref{liminf}) and (\ref{limsup}), it follows that $v$ is continuous at $z_0$. As $z_0$ was arbitrarily chosen, $v$ is continuous in $D$.

\medskip

To see that $v$ is subharmonic in $D$, again choose and fix $z_0 \in D$. Let $h \in \mathscr{F}_{z_0}$ be the harmonic extension of $u$ on a domain $W$ (i.e. the Poisson integral of the $L^\infty$ function $u$ on $\partial W$) such that $v(z_0)=h(z_0)$. Then for all balls $B$ compactly contained in $W$, by (\ref{v_bigger_h}) we have
\begin{equation*}
v(z_0)=h(z_0) = \int_{\partial B} h(y) \:d\mu_{z_0,B}(y) \le \int_{\partial B} v(y) \:d\mu_{z_0,B}(y) \:. 
\end{equation*}
Since subharmonicity is a local property, it follows that $v$ is subharmonic in $D$.
\end{proof}
Our next Lemma shows that if $u$ has nice boundary behaviour then so does $v$.
\begin{lemma} \label{v_bdry_bhv}
Under the hypothesis of Theorem \ref{domain_bdry_bhv} or Theorem \ref{neg_bdry_bhv}, if $\xi \in \partial D$ such that 
\begin{equation*}
\displaystyle\lim_{\substack{z \to \xi\\ z \in D}} u(z) = u_\xi \:, \text{then }
\displaystyle\lim_{\substack{z \to \xi\\ z \in D}} v(z) = u_\xi \:.
\end{equation*}
\end{lemma}
\begin{proof}
Let $\xi \in \partial D$ such that $u$ has the mentioned boundary behaviour at $\xi$. Let $\{z_n\}_{n=1}^\infty \subset D$ such that $z_n \to \xi$ as $n \to \infty$.  We will show that
\begin{equation} \label{aim}
v(z_n) \to u_{\xi} \text{ as } n \to \infty \:.
\end{equation}
This is seen as follows. By Lemma \ref{sup_attained}, there exist domains $W_n$ such that $W_n \subset D$ and the harmonic extensions $h_n \in \mathscr{F}_{z_n}$ of $u$ on $W_n$ (i.e. the Poisson integrals of the $L^\infty$ function $u$ on $\partial W_n$) satisfy $v(z_n)=h_n(z_n)$. We note that 
to show (\ref{aim}), it is enough to show that each subsequence of $\{v(z_n)\}_{n=1}^\infty$ has a further subsequence which converges to $u_{\xi}$\:. So let $\{v(z_{n_k})\}_{k=1}^\infty$ be a given subsequence, which after relabelling we write as $\{v(z_n)\}_{n=1}^\infty$\:.

\medskip
 
Now in the setting of Theorem \ref{domain_bdry_bhv}, there are two cases:
\begin{itemize}
\item[(i)] diameters of $W_n$ converge to $0$,
\item[(ii)] diameters of $W_n$ do not converge to $0$.
\end{itemize}

\medskip

In case $(i)$, the domains $W_n$ are balls and $\overline{W_n} \to \{\xi\}$ in the Hausdorff topology on compacts in $\overline{D}$. Now as the radii of these balls are converging to $0$, the second fundamental forms of the $\partial W_n$'s are going to $+\infty$, whereas the second fundamental form of $\partial D$ is bounded, hence for $n$ large, $\partial W_n$ can intersect $\partial D$ at most at one point (which has zero harmonic measure). Then by the boundary behaviour of $u$ at $\xi$, we have 
\begin{eqnarray*}
|v(z_n)- u_{\xi}|
& \le & \int_{\partial W_n} |u(y)-u_{\xi}| \:d\mu_{z_n,W_n}(y) \\
& \le & \sup_{y \in D \cap \overline{B(\xi, \epsilon_n)}} |u(y)-u_{\xi}| \: \to 0 \text{  as } n \to \infty \:.
\end{eqnarray*}
(where $\epsilon_n = d(z_n,\xi)+diam(W_n) \to 0$ as $n \to \infty$).

\medskip

In case $(ii)$ by passing to a subsequence, we may assume that $\overline{W_n} \to \overline{W}$ and $\partial W_n \to \partial W$ (for some non-degenerate domain $W$ such that $W \subset D$) in the Hausdorff topology on compacts in $\overline{D}$. Then using the boundary behaviour of $u$ at $\xi$, given $\varepsilon>0$, there exists $\delta >0$ such that 
\begin{equation} \label{v_bdry_bhv_eq1}
|u(y)-u_\xi| < \varepsilon/2 \:,\text{ for all } y \in D \cap B(\xi,\delta) \:.
\end{equation}
Set $C:= \displaystyle\sup_{D}|u|$. Now if for all $n$ large,
\begin{equation} \label{v_bdry_bhv_eq2}
\mu_{z_n,W_n}(\partial W_n \cap \partial D)=0\:,
\end{equation}
then for $n$ large we have, by (\ref{v_bdry_bhv_eq1}) and Lemma \ref{har_meas_domains},
\begin{eqnarray} \label{v_bdry_bhv_eq3}
|v(z_n) - u_{\xi}| & = & \left| \int_{\partial W_n \cap D} (u(y) - u_{\xi}) \  d\mu_{z_n, W_n}(y) \right| \nonumber\\
 & \le & 
\int_{\partial W_n \setminus B(\xi,\delta)} |u(y)-u_\xi|\:d\mu_{z_n,W_n}(y) \nonumber\\
& + &
\int_{\partial W_n \cap B(\xi,\delta) \cap D} |u(y)-u_\xi|\:d\mu_{z_n,W_n}(y) \nonumber\\
& \le & 2C \mu_{z_n,W_n}(\partial W_n \setminus B(\xi,\delta)) + \frac{\varepsilon}{2} \nonumber \\
&<& \varepsilon \:.  
\end{eqnarray}
On the other hand if (\ref{v_bdry_bhv_eq2}) is not true for $W_n$, then we can consider balls $V_n$ such that 
\begin{itemize}
\item $z_n \in V_n$,
\item $\partial V_n \cap \partial D = \emptyset$\:,
\item the Hausdorff distance between the closures $\overline{V_n}$ and $\overline{W_n}$, as well as the boundaries $\partial V_n$ and $\partial W_n$ are sufficiently small so that for $\tilde{h}_n$, the harmonic extensions of $u$ on $V_n$, one has by Lemma \ref{poisson_conv_2},
\begin{equation} \label{v_bdry_bhv_eq4}
|\tilde{h}_n(z_n)-h_n(z_n)| < \frac{1}{n} \:,
\end{equation}
\item $\overline{V_n} \to \overline{W}$ and $\partial V_n \to \partial W$ in the Hausdorff topology on compacts in $\overline{D}$. 
\end{itemize}
Then proceeding as in (\ref{v_bdry_bhv_eq3}) for the domains $V_n$, we get for $n$ large enough,
\begin{equation} \label{v_bdry_bhv_eq5}
|\tilde{h}_n(z_n)-u_{\xi}| < \varepsilon \:.
\end{equation}
Finally, combining (\ref{v_bdry_bhv_eq4}) and (\ref{v_bdry_bhv_eq5}) we have for $n$ large,
\begin{equation*}
|v(z_n)-u_{\xi}| \le |h_n(z_n) - \tilde{h}_n(z_n)| + |\tilde{h}_n(z_n)-u_{\xi}| < 2 \varepsilon \:.
\end{equation*}
This completes the proof of (\ref{aim}) in the setting of Theorem \ref{domain_bdry_bhv}.

\medskip

Then in the setting of Theorem \ref{neg_bdry_bhv}, there are again two cases:
\begin{itemize}
\item[(i)] diameters of $W_n$ are bounded,
\item[(ii)] diameters of $W_n$ are unbounded. 
\end{itemize}

\medskip

In case $(i)$, the domains $W_n$ are balls and as their diameters are uniformly bounded it follows that $\overline{W_n} \to \{\xi\}$
in the cone topology. Then using the boundary behaviour of $u$ at $\xi$ and proceeding as in case $(i)$ in the scenario of bounded domains, we get the result.

\medskip

Similarly, for case $(ii)$, using the boundary behaviour of $u$ at $\xi$, Lemma \ref{har_meas_compl_cone}, Lemma \ref{poisson_conv_2} and proceeding as in case $(ii)$ in the scenario of bounded domains, we get the result.
\end{proof}

\begin{remark} \label{remark}
One may also define $w(z):= \displaystyle\inf_{h \in \mathscr{F}_z} h(z)$, for all $z \in D$. Proceeding similarly as above it can be shown that 
\begin{itemize}
\item for all $z \in D$, there exists $h \in \mathscr{F}_z$ such that $w(z)=h(z)$,
\item $w$ is continuous, superharmonic in $D$ satisfying $w(z) \le u(z)$, for all $z \in D$ and 
\item under the hypothesis of Theorem \ref{domain_bdry_bhv} or Theorem \ref{neg_bdry_bhv}, if $\xi \in \partial D$ such that 
\begin{equation*}
\displaystyle\lim_{\substack{z \to \xi\\ z \in D}} u(z) = u_\xi \:, \text{then }
\displaystyle\lim_{\substack{z \to \xi\\ z \in D}} w(z) = u_\xi \:.
\end{equation*}
\end{itemize}
This $w$ will be useful for us.
\end{remark}

\section{Some $L^\infty$ maximum principles}
In this section we prove some useful maximum principles. The first one is for pre-compact domains $D$.
\begin{lemma} \label{max_princ1}
Let $\varphi$ be a bounded, continuous, real-valued subharmonic function on $D$ such that  
\begin{equation*}
\displaystyle\lim_{\substack{z \to \xi\\ z \in D}} \varphi(z) = f(\xi) \:,
\end{equation*}
for almost every boundary point $\xi$, with respect to the harmonic measures on $\partial D$, where $f \in L^\infty(\partial D)$. Then 
\begin{equation*}
\varphi(x) \le P[f](x) \:,
\end{equation*}
for all $x \in D$, where $P[f]$ is the Poisson integral of $f$.
\end{lemma}
We also have an analogue of \ref{max_princ1} for manifolds with pinched negative curvature.
\begin{lemma} \label{max_princ2}
Let $\varphi$ be a bounded, continuous, real-valued subharmonic function on $M$ (where $M$ is as in Theorem \ref{neg_bdry_bhv}) such that  
\begin{equation} \label{bdry_bhv_max_princ}
\displaystyle\lim_{\substack{z \to \xi\\ z \in M}} \varphi(z) = f(\xi) \:,
\end{equation}
for almost every $\xi \in \partial M$, with respect to the harmonic measures on $\partial M$, where $f \in L^\infty(\partial M)$. Then 
\begin{equation} \label{max_princ_ineq}
\varphi(x) \le P[f](x) \:,
\end{equation}
for all $x \in M$, where $P[f]$ is the Poisson integral of $f$.
\end{lemma}
The proofs of Lemma \ref{max_princ1} and \ref{max_princ2} are similar. We only give the proof of Lemma \ref{max_princ2}.
\begin{proof}[Proof of Lemma \ref{max_princ2}]

Let $x \in M$, and for $R > 0$ let $\tau_R$ denote the first exit time of Brownian motion started at $x$ from the ball $B(x, R)$. Since Brownian paths are continuous, it follows that
\begin{equation*}
\tau_R \to \infty \text{ almost surely as } R \to \infty.
\end{equation*}

\medskip

Recalling Sullivan's result (\cite{Sullivan}) that almost surely the Brownian motion $(B_t)_{t \geq 0}$ converges to a point $B_{\infty}$ on the boundary 
$\partial M$, it follows that $B_{\tau_R} \to B_{\infty}$ almost surely as $R \to \infty$.  

\medskip

Let $A \subset \partial M$ be a set of full harmonic measure, $\mu_x(A) = 1$, 
such that $\phi(z) \to f(\xi)$ as $z \to \xi$ for all $\xi \in A$. Since the harmonic measure $\mu_x$ on $\partial M$ is the distribution of the random variable $B_{\infty}$, it follows that $B_{\infty} \in A$ almost surely. 
Hence almost surely we have $\phi(B_{\tau_R}) \to f(B_{\infty})$ as $R \to \infty$. 

\medskip

Since $\phi$ is bounded, the random variables $\{ \phi(B_{\tau_R}) \}_{R > 0}$ are uniformly bounded, thus by the Dominated convergence Theorem we have

\begin{equation} \label{conv_of_expectations}
\E(\varphi(B_{\tau_R})) \to \E(f(B_\infty)) \text{ as } R \to \infty.
\end{equation}



\medskip

On the other hand, since $\varphi$ is subharmonic in $M$, we have the sub-mean value property for geodesic spheres, hence 
\begin{equation} \label{smvp_expectation}
\varphi(x) \le \E(\varphi(B_{\tau_R}))\:, \text{ for all } R>0\:. 
\end{equation}

\medskip

It follows from (\ref{conv_of_expectations}) and (\ref{smvp_expectation}) that

\begin{equation*}
\varphi(x) \le \E(f(B_\infty)) = \int_{\partial M} f(\xi) d \mu_o(\xi) = P[f](x) \:
\end{equation*}

as required. 

%
%
%
%

%
\end{proof}

\section{Proofs of the main results} 
\begin{proof}[Proof of Theorem \ref{domain_bdry_bhv}]
By Lemma \ref{v_bdry_bhv},
\begin{equation*}
\displaystyle\lim_{\substack{z \to \xi\\ z \in D}} v(z) = u_\xi \:, \text{ for almost every } \xi \in \partial D\:.
\end{equation*}
By remark \ref{remark} we also have,
\begin{equation*}
\displaystyle\lim_{\substack{z \to \xi\\ z \in D}} w(z) = u_\xi \:, \text{ for almost every } \xi \in \partial D\:.
\end{equation*}

Let $\varphi:= v-w$. Then as by construction, $w \le u \le v$ and by the maximum principle and boundedness of $u$, both $v$ and $w$ are bounded, it follows that $\varphi$ is a bounded, non-negative, continuous, subharmonic function in $D$ such that 
\begin{equation*}
\displaystyle\lim_{\substack{z \to \xi\\ z \in D}} \varphi(z) = 0 \:, \text{ for almost every } \xi \in \partial D\:.
\end{equation*} 

Then by Lemma \ref{max_princ1},
\begin{equation*}
\varphi \le 0 \: \text{ on } D \:.
\end{equation*}

So $\varphi$ vanishes identically on $D$. Therefore $u,v$ and $w$ are identically equal to each other on $D$. Hence $u$ is both subharmonic and superharmonic, that is, harmonic in $D$.
\end{proof}

\begin{proof}[Proof of Theorem \ref{domain_radius}]
Let $R \in (\tau,1)$ and let $z \in D$ be such that
\begin{equation*}
d(z, \partial D) \le (1-(R + \kappa(1-R)))r_D \:.
\end{equation*}

Then by Lemma \ref{sup_attained}, there exists a ball $B(x,r) \subset D$ such that the harmonic extension $h \in \mathscr{F}_z$ of $u$ on $B(x,r)$  satisfies $v(z)=h(z)$. Moreover by (\ref{rho_0}),
\begin{eqnarray*}
d(x, \partial D) &\le & d(z, \partial D) + d(z, x) \\
& \le & (1-(R + \kappa \:(1-R)))r_D + \rho_0(x) \\
& \le & (1-(R + \kappa \:(1-R)))r_D + \kappa \:d(x, \partial D) 
\end{eqnarray*}

Then an elementary computation along with the fact that $\kappa \in (0,1)$ implies that
\begin{equation} \label{domain_radius_eq1}
d(x, \partial D) \le (1-R)r_D \:.
\end{equation}

Now as $R \in (\tau,1)$, from above it follows that
\begin{equation*}
d(x, \partial D) \le (1-R)r_D < (1-\tau)r_D\:, 
\end{equation*}
and hence by (\ref{rho_0}) we also have 
\begin{equation} \label{domain_radius_eq2}
d(z,x) < r \le \rho_0(x) \le \delta(r(x), \varepsilon(d(x, \partial D));u)\:.
\end{equation}

Then using (\ref{domain_radius_eq1}) and (\ref{domain_radius_eq2}) we get,
\begin{eqnarray*}
0 & \le & v(z) - u(z) \\
& \le & |h(z) - u(x)| + | u(x) - u(z)| \\
& \le & \int_{\partial B(x,r)} |u(y)- u(x)| \:d\mu_{z,B(x,r)}(y)  + | u(x) - u(z)| \\
& \le & \sup_{y \in \partial B(x,r)} |u(y)- u(x)|   + | u(x) - u(z)| \\
& \le & 2 \:\varepsilon(d(x, \partial D)) \\
& \le & 2 \sup_{t \in (0,(1-R)r_D]} \varepsilon(t) \:\to 0 \:,\text{ as } R \to 1 \:.
\end{eqnarray*}

Thus 
\begin{equation*}
(v-u)(z) \to 0 \text{ uniformly as }d(z, \partial D) \to 0\:.
\end{equation*}
Using remark \ref{remark} and proceeding as above we also have,
\begin{equation*}
(u-w)(z) \to 0 \text{ uniformly as }d(z, \partial D) \to 0\:.
\end{equation*}
Thus $\varphi:=v-w$ is a non-negative, 
continuous subharmonic function in $D$ such that 
\begin{equation*}
\varphi(z) \to 0 \text{ uniformly as }d(z, \partial D) \to 0\:.
\end{equation*}
This implies that $\varphi$ is also bounded in $D$. Then by Lemma \ref{max_princ1},
\begin{equation*}
\varphi \le 0 \text{ on } D\:.
\end{equation*}
The remainder of the proof is as in the final paragraph of the proof of Theorem \ref{domain_bdry_bhv}.
\end{proof}

\begin{proof}[Proof of Theorem \ref{neg_bdry_bhv}]
By Lemma \ref{v_bdry_bhv},
\begin{equation*} 
\displaystyle\lim_{\substack{z \to \xi\\ z \in M}} v(z) = u_\xi \:\: \text{for almost every }\xi \in \partial M\:.
\end{equation*}

Then using remark \ref{remark} and proceeding as in the proof of Theorem \ref{domain_bdry_bhv}, we get that $\varphi:= v-w$, is a bounded, non-negative, continuous, subharmonic function in $M$ such that
\begin{equation*} 
\displaystyle\lim_{\substack{z \to \xi\\ z \in M}} \varphi(z) = 0 \:\: \text{for almost every }\xi \in \partial M\:.
\end{equation*}
Then by Lemma \ref{max_princ2},
\begin{equation*}
\varphi \le 0 \text{ on } M \:.
\end{equation*}
The remainder of the proof is as in the final paragraph of the proof of Theorem \ref{domain_bdry_bhv}. 
\end{proof}

\section{Final remarks}

We conclude with some remarks and questions.

\begin{enumerate}

\item Functions satisfying the RMVP in a domain $D \subset M$ can be interpreted as functions which are harmonic with respect to a certain random walk $(X_n)_{n \geq 1}$ on $D$. Namely, the distribution of $X_{n+1}$ given that $X_n = z \in D$ is the harmonic measure $\mu_{z, B}$ on the sphere $\partial B$, where $B$ is the geodesic ball $B = B(z, \rho(z))$. Theorem \ref{domain_bdry_bhv} can then be interpreted as saying that continuous harmonic functions for the random walk with the appropriate boundary behaviour are also harmonic for the Laplace-Beltrami operator. 
   
\item Theorem \ref{neg_bdry_bhv} assumes that the manifold $M$ has strictly negative curvature. It would be interesting to study whether an analogue of this result holds for nonpositively curved manifolds, at least those which are Gromov hyperbolic. Such a result would cover the case of harmonic manifolds of purely exponential volume growth (which are nonpositively curved and Gromov hyperbolic), in particular it would cover all known examples of nonflat noncompact harmonic manifolds, namely the rank one symmetric spaces of noncompact type and the Damek-Ricci spaces.

\item Finally, for all the results obtained, it would be interesting to see if there are counterexamples if any of the hypotheses are relaxed. Constructing functions on a general Riemannian manifold which are not harmonic but satisfy the RMVP seems to be a difficult task however.   
\end{enumerate}

\section*{Acknowledgements} The authors  would  like to thank Swagato K. Ray for suggesting the problems. The second author is supported by a Research Fellowship of Indian Statistical Institute. 

\bibliographystyle{amsplain}

\end{document}